\documentclass[12pt]{amsart}
\usepackage{amssymb}
\usepackage{amsmath}
\usepackage[active]{srcltx}
\usepackage{t1enc}
\usepackage[latin2]{inputenc}
\usepackage{verbatim}
\usepackage{amsmath,amsfonts,amssymb,amsthm}
\usepackage[mathcal]{eucal}
\usepackage{enumerate}
\usepackage[centertags]{amsmath}
\usepackage{graphics}

\newtheorem{theorem}{Theorem}

\newtheorem{lemma}{Lemma}

\newtheorem{Conjecture}{Conjecture}
\newtheorem{Proposition}{Proposition}
\newtheorem{Problem}{Problem}

\begin{document}
\author{G. Tephnadze}
\title[convergence]{Convergence and Strong Summability of the Two--dimensional Vilenkin-Fourier Series}
\address{G. Tephnadze, The University of Georgia, School of Science and Technology, 77a Merab Kostava St, Tbilisi, 0128, Georgia}
\email{g.tephnadze@ug.edu.ge }
\date{}

\thanks{The research was supported by Shota Rustaveli National Science Foundation grant YS-18-043.}
\maketitle
	
\begin{abstract}
In this paper we investigate convergence and strong summability of the two-dimensional Vilenkin-Fourier series in the martingale Hardy spaces.
\end{abstract}
	
	
	
\section{Introduction}
	
The definitions and notations used in this introduction can be found in our
next Section. It is known \cite[p. 125]{G-E-S} that the two-dimensional Vilenkin systems are not Schauder bases in $L_1\left(G^2_m\right).$ Moreover, (see \cite{AVD} and \cite{S-W-S}) there exists a function $f\in H_{p}^{\square }\left(G_m^2\right),$ such that the corresponding partial sums are not bounded in $L_p\left(G_m^2\right),$  for all $0<p\leq 1.$ 

However, Weisz \cite{We} proved that
if $\alpha \geq 0$, $0<p\leq1$ and $f\in H_p\left(
G^2_m\right),$ then there exists an absolute constant $c_{p},$ depending only
on $p,$ such that
\begin{equation} \label{thtwo}
\underset{n,m\geq 2}{\sup }\left( \frac{1}{\log n\log m}\right)^{\left[ p\right] }\underset{2^{-\alpha }\leq k/l\leq 2^{\alpha },\text{ }\left(k,l\right) \leq \left( n,m\right) }{\sum }\frac{\left\Vert
S_{k,l}f\right\Vert _{p}^{p}}{\left( kl\right) ^{2-p}}\leq c_{p}\left\Vert f\right\Vert_{H_p^{\square}\left(G^2_m\right)}^{p},
\end{equation}
where $\left[ p\right] $ denotes the integer part of $p.$ Moreover, in \cite{tep9} it was proved that the rate of sequence $\left(kl\right)^{2-p} \ (0<p<1)$ in inequality (\ref{thtwo})  can not be improved, which gives sharpness for $\alpha>0$.

In the case when $\alpha =0$ and $0<p\leq 1$ it follows that if $f\in H_{p}\left( G^{2}_m\right),$ then there exists
an absolute constant $c,$ such that 
\begin{equation} \label{1}
\underset{n\geq 2}{\sup }\frac{1}{\log ^{2[p]}n}\underset{k=0}{\sum^{n}}\frac{\left\Vert
S_{k,k}f\right\Vert_p^p}{k^{4-2p}}\leq c\left\Vert f\right\Vert _{H_{p}^{\square}\left(G^2_m\right) }.  
\end{equation}

For the two-dimensional Walsh-Fourier series Goginava and Gogoladze \cite{gg} for $p=1$ and Tephnadze \cite{tep1} for $0<p<1$ generalized inequality \eqref{1} and proved that there exists an absolute constant $c_{p}$, depending only on $p,$ such that
\begin{equation}\label{2cc}
\sum\limits_{n=1}^{\infty }\frac{\left\Vert S_{n,n}f\right\Vert _{p}}{%
n^{3-2p}\log ^{2\left[ p\right] }\left( n+1\right) }\leq c_{p}\left\Vert f\right\Vert_{H_{p}^{\square}\left(G^2_m\right) },  
\end{equation}%
for all $f\in H_{p}^{\square }\left( G_{m}^{2}\right) $, where $\left[ p\right] $ denotes integer part of $p.$
Moreover, in \cite{tep1} and \cite{tep2} there were proved that the sequence $\left\{ 1/k^{3-2p}\log ^{2\left[ p\right] }\left( k+1\right) :k\in \mathbb{N}\right\} $ in inequality (\ref{2cc}) is sharp.

Strong summability of the one and two-dimensional Vilenkin-Fourier (Walsh-Fourier) series  can be found in Baramidze, Persson, Tephnadze and Wall \cite{BPTW}, Blahota \cite{b}, \cite{BNPT}, Belinskii \cite{Be1}, Gát  \cite{gat1}, Goginava and Gogoladze \cite{GG1},
Smith \cite{sm}, Simon \cite{si1}, \cite{Si2},   Blahota and Tephnadze  \cite{B-T}, \cite{B-T-T},  Tephnadze  \cite{tep5}, Tutberidze \cite{tut1}.

Concerning approximation properties of Fourier series in the classical and real Hardy spaces only few results are known. We refer to the papers by Oswald \cite{Os}, Kryakin and Trebels \cite{KT}, Storoienko \cite{St1}, \cite{St2}. For martingale Hardy spaces approximation properties of some general summability methods were investigated in Fridli, Manchanda and Siddiqi \cite{FMS} (see also \cite{4}, \cite{FR}), Tephnadze \cite{tep4}, \cite{tep6}, \cite{tep7}, \cite{tep8}, Nagy \cite{n}, \cite{n1}, \cite{nagy}, Weisz \cite{We2}, \cite{we4}. In \cite{ptw2} it was proved that if $0<p<1,$  $2^{-\alpha }\leq m/n\leq 2^{\alpha }$ and
\begin{equation*} \label{10Atwo}
\omega_{H_p^{\square}(G^2_m)} \left( \frac{1}{2^k},f\right)=o\left(\frac{1}{2^{
k(2/p-2) }}\right),\text{ as \ }k\rightarrow \infty,
\end{equation*}
then
\begin{equation*}
\left\Vert S_{m,n}f-f\right\Vert_{H_p^{\square}(G^2_m)}\rightarrow
0,\text{ as }m,n\rightarrow\infty .
\end{equation*}

Moreover, there exists a martingale $f\in H_{p}(G^2_m),$ such that 
\begin{equation*}
\omega_{H_p(G^2_m)} \left(\frac{1}{2^{k}},f\right) =O\left( \frac{1}{{2}^{k(2/p-2)}}\right),\text{ as \ }k\rightarrow \infty
\end{equation*}
and
\begin{equation*}
\left\Vert S_{m,n}f-f\right\Vert_{weak-L_p(G^{2}_m)}\nrightarrow
0 \ \text{ as  }m,n\rightarrow \infty .
\end{equation*}
$0<p<1$ and $2^{-\alpha}<m/n\leq 2^{\alpha}.$
	
The main aim of this paper is to generalize inequality (\ref{thtwo})  for bounded Vilenkin systems. We also prove that the sequence $\left( kl\right)^{2-p}$ in inequality (\ref{thtwo}) can not be improved. Moreover, we find necessary and sufficient conditions for modulus of continuity of the two-dimensional Vilenkin-Fourier series, which provide convergence of partial sums in $H^{\square}_p(G^{2}_m)$ norm.
	
This paper is organized as follows: in order not to disturb our discussions later on some definitions and notations and some important propositions are presented in Section 2. The main results with proofs can be found in Section 3. Moreover, in Section 4 we will state some interesting open problems and conjectures of this research area.
	
\section{Preliminary}
	
Denote by $\mathbb{N}_{+}$ the set of positive integers, $\mathbb{N}:=%
\mathbb{N}_{+}\cup \{0\}.$ Let $m:=(m_{0,}$ $m_{1},...)$ be a sequence of positive integers not less than 2. Denote by
$Z_{m_{k}}:=\{0,1,\ldots ,m_{k}-1\}$
the additive group of integers modulo $m_{k}$.
	
Define the group $G_{m}$ as the complete direct product of the groups $Z_{m_{i}}$ with the product of the discrete topologies of $Z_{m_{j}}`$s.
	
In this paper we discuss bounded Vilenkin groups,\ i.e. in the case when $\sup_{n\in \mathbb{N}}m_{n}<\infty.$
	
The direct product $\mu $ of the measures
\begin{equation*}
\mu _{k}\left( \{j\}\right) :=1/m_{k},\ (j\in Z_{m_{k}})
\end{equation*}%
is the Haar measure on $G_{m_{\text{ }}}$ with $\mu \left( G_{m}\right) =1.$
	
The elements of $G_{m}$ are represented by sequences
\begin{equation*}
x:=\left( x_{0},x_{1},\ldots ,x_{j},\ldots \right) ,\ \left( x_{j}\in
Z_{m_{j}}\right) .
\end{equation*}
	
It is easy to give a base of neighbourhoods of $G_{m}:$
\begin{equation*}
\text{ }I_{0}\left( x\right) :=G_{m},\text{ \ }I_{n}(x):=\{y\in G_{m}\mid
y_{0}=x_{0},\ldots ,y_{n-1}=x_{n-1}\}\,\,\left( x\in G_{m},\text{ }n\in
\mathbb{N}\right) .
\end{equation*}
	
Denote $I_{n}:=I_{n}\left( 0\right) ,$ for $n\in \mathbb{N}_{+}$ and
$
e_{n}:=\left( 0,\ldots ,0,x_{n}=1,0,\ldots \right) \in G_{m},\text{ \ }\left( n\in \mathbb{N}\right) .
$
	
It is evident that
\begin{equation} \label{2}
\overline{I_{N}}=\overset{N-1}{\underset{s=0}{\bigcup }}I_{s}\backslash
I_{s+1}.  
\end{equation}

If we define the so-called generalized number system based on $m$ in the following way :
\begin{equation*}
M_{0}:=1,\ M_{k+1}:=m_{k}M_{k}\,\,\,\ \ (k\in \mathbb{N}),
\end{equation*}%
then every $n\in \mathbb{N}$ can be uniquely expressed as
\begin{equation*}
n=\sum_{j=0}^{\infty }n_j M_j,\text{ \ where \ \ }n_j\in Z_{m_j}\text{\ }(j\in \mathbb{N})
\end{equation*}%
and only a finite number of $n_{j}`$s differ from zero. Let $|n|$ denote the
largest integer $j$ for which $n_{j}\neq 0$.
	
Denote by $\mathbb{N}_{n_{0}}$ the subset of positive integers $\mathbb{N}_{+},$ for which $n_{0}=1.$ Then for every $n\in \mathbb{N}_{n_{0}},$ $M_{k}<n<$ where $M_{k+1}$ can be written as
\begin{equation*}
n=1+\sum_{j=1}^k n_j M_j,
\end{equation*}
where
$n_{j}\in \left\{ 0,\ldots {},m_{j}-1\right\}, \ \ \ (j=1,\ldots,k-1)\text{ and \ }n_{k}\in \left\{ 1,\ldots {},m_{k}-1\right\}.$
	
For any $\alpha>0$ we get that
\begin{equation} \label{1atwo}
\underset{\left\{ n:M_{k}\leq n\leq 2^{\alpha}M_{k},\text{ }n\in \mathbb{N}_{n_{0}}\right\} }{\sum }1=\frac{2^{\alpha}M_{k}-M_{k}}{m_{0}}\geq cM_{k},
\end{equation}%
where $c$ is an absolute constant.
	
The norms (quasi-norm) of the spaces $L_{p}(G_{m}^{2})$ and $weak-L_{p}\left( G_{m}^{2}\right) $  are respectively defined by 
\begin{equation*}
\left\Vert f\right\Vert _{p}:=\left( \int_{G_{m}^{2}}\left\vert f\right\vert^{p}d\mu \times d\mu \right) ^{1/p},\ \ \ \ \ \ \  \left\Vert f\right\Vert _{weak-L_{p}}:=\underset{\lambda >0}{\sup }\lambda
\mu \left( f>\lambda \right) ^{1/p} \ \ \ \ \ \ 
\left( 0<p<\infty \right) .
\end{equation*}
	
Next, we introduce on $G^2_m$ an orthonormal system which is called the
Vilenkin system.
	
At first, we define the complex-valued function $r_{k}\left( x\right)
:G_{m}\rightarrow \mathbb{C},$ the generalized Rademacher functions by
\begin{equation*}
r_{k}\left( x\right) :=\exp \left( 2\pi ix_{k}/m_{k}\right),\text{ }\left(i^{2}=-1,x\in G_{m},\text{ }k\in \mathbb{N}\right).
\end{equation*}
	
Now, define the Vilenkin system $\psi :=(\psi_{n}:n\in\mathbb{N})$ on $G_{m}$ as:
\begin{equation*}
\psi _{n}(x):=\prod\limits_{k=0}^{\infty }r_{k}^{n_{k}}\left( x\right)
,\ \ \left( n\in\mathbb{N}\right).
\end{equation*}

We define the two-dimesional Vilenkin system as a kronecker product of two Vilenkin systems. The Vilenkin system is orthonormal and complete in $L_2\left( G_m^2\right)$ (for details see \cite{AVD} and \cite{Vi}).
Specifically, we call this system the Walsh-Paley system, when $m\equiv 2.$

The rectangular partial sum of the 2-dimensional Vilenkin-Fourier series of function $f\in L_{2}\left( G_{m}^{2}\right) $ is defined as follows:
\begin{equation*}
S_{M,N}f\left( x,y\right) :=\sum\limits_{i=0}^{M-1}\sum\limits_{j=0}^{N-1}%
\widehat{f}\left( i,j\right) \psi _{i}\left( x\right) \psi _{j}\left(
y\right) ,
\end{equation*}
where the numbers
\begin{equation*}
\widehat{f}\left( i,j\right) =\int\limits_{G_m^2}f\left(x,y\right) \bar{\psi}_i\left(x\right)\bar{\psi}_{j}\left( y\right) d\mu \left(x\right)d\mu\left(y\right)
\end{equation*}
is said to be the $\left( i,j\right) $-th Vilenkin-Fourier coefficient of the function \thinspace $f.$
	
It is well-known that (for details see e.g. \cite{S-W-S}) 
\begin{equation*}
S_{M,N}f\left( x,y\right)=\int_{G^2_m}f(x,y)D_M(x-t)D_N(y-s)d\mu \left(x\right)d\mu \left(y\right),
\end{equation*}
where
$$D_n(x)=\sum\limits_{i=0}^{n-1}\psi_i(x)$$
is called as $n$-th Dirichlet Kernel. Recall that
\begin{equation} \label{1dn}
D_{M_n}\left(x\right)=\left\{
\begin{array}{ll}
M_{n}, & \text{if}\ \ x\in I_n, \\
0, & \text{if}\ \ x\notin I_n.
\end{array}%
\right.  
\end{equation}

It is also known that (see \cite{AVD})
\begin{equation} \label{2dn}
D_{sM_{n}}=D_{M_{n}}\sum_{k=0}^{s-1}\psi
_{kM_{n}}=D_{M_{n}}\sum_{k=0}^{s-1}r_{n}^{k} \ \ \ \ \ \ \
\text{and} \ \ \ \ \ \ \
D_n=\psi_n\left( \sum_{j=0}^{\infty}D_{M_j} \sum_{u=m_j-n_j}^{m_j-1}r_j^u\right).  
\end{equation}

Moreover, the following estimation holds true: 
\begin{equation} \label{dn2.6}
\int_{I_{N}}\left\vert D_{n}\left( x-t\right) \right\vert d\mu \left(
t\right) \leq \frac{cM_{s}}{M_{N}}, \ \ \ \ \ x\in I_{s}\backslash I_{s+1}, \ \ s=0,...,N-1.
\end{equation}

For our investigation we need the following estimates of the two-dimensional Dirichlet kernels of independent interest:
\begin{lemma}\label{dntwo2.1}\bigskip Let $m,n\in \mathbb{N}$. Then, for every $ 0<\varepsilon\leq 1$, there exists an absolute constant $c,$ such that
	\begin{equation*}
	\int_{I_{N}\times I_{N}}\left\vert D_{m}\left(x-t\right)D_{n}\left(
	y-s\right) \right\vert d\mu \left( t\right)d\mu \left( s\right) \leq \frac{cm^{\varepsilon}M_{s}}{M_N^{1+\varepsilon}}, \ \ \ (x,y)\in {I_N}\times\left(I_{s}\backslash
	I_{s+1}\right), \ \ s=0,...,N-1
	\end{equation*}
	and 
	\begin{equation*}
	\int_{I_{N}\times I_{N}}\left\vert D_{m}\left( x-t\right)D_{n}\left(
	y-s\right) \right\vert d\mu \left( t\right)d\mu \left( s\right) \leq \frac{cn^{\varepsilon}M_{s}}{M_{N}^{1+\varepsilon}}, \ \ \ (x,y)\in \left(I_{s}\backslash
	I_{s+1}\right)\times{I_N}, \ \ \ s=0,...,N-1.
	\end{equation*}
	
	Moreover, let $(x,y)\in \left(I_{s_1}\backslash
	I_{s_1+1}\right)\times \left(I_{s_2}\backslash I_{s_2+1}\right),$ $%
	s_1,s_2=0,...,N-1.$ Then there exists an absolute constant $c,$ such that
	\begin{equation*}
	\int_{I_{N}}\left\vert D_{n}\left( x-t\right)D_{m}\left( y-s\right)
	\right\vert d\mu \left( t\right)d\mu \left( s\right) \leq \frac{cM_{s_1}M_{s_2}}{M_{N}^2}.
	\end{equation*}
\end{lemma}

\begin{proof} 
	Since $\left\vert D_m (x)\right\vert \leq m$ and 
	$\left\vert D_m(x) \right\vert \leq M_{s}, \ \ \text{for } \ \ x\in I_s \backslash I_{s+1},$ by using  \eqref{dn2.6} we obtain that
\begin{eqnarray*}
&&\int_{I_{N}}\left\vert D_{m}\left( x-t\right) \right\vert d\mu
\left(t\right)\leq m^{\varepsilon}\overset{\infty}{\underset{s=N}{
\sum }} \int_{I_{s}\backslash I_{s+1}}\left\vert D_{m}\left( x-t\right) \right\vert ^{1-\varepsilon}d\mu \left( t\right) \\
&\leq&c m^{\varepsilon}\overset{\infty}{\underset{s=N}{
\sum }}\int_{I_{s}\backslash I_{s+1}}M_{s}^{1-\varepsilon}d\mu \left( t\right) \leq c m^{\varepsilon}\overset{\infty}{\underset{s=N}{\sum }}M_{s}^{-\varepsilon}\leq\frac {cm^{\varepsilon}}{M_{N}^{\varepsilon}}.
	\end{eqnarray*}
	
Therefore, by using inequality \eqref{dn2.6} we obtain that 
\begin{eqnarray*}
&&\int_{I_{N}\times I_{N}}\left\vert D_{m}\left( x-t\right)D_{n}\left(
y-s\right) \right\vert d\mu \left(t\right)d\mu \left(s\right) \\
&\leq & \int_{I_N}\left\vert D_m\left(x-t\right) \right\vert d\mu
\left(t\right)\int_{I_N}\left\vert D_{n}\left( y-s\right) \right\vert d\mu\left(s\right)\leq  \frac{cm^{\varepsilon}M_{s}}{M_{N}^{1+\varepsilon}}.
	\end{eqnarray*}

The proof of the second estimation is quite analogical to the proof of Lemma \ref{dntwo2.1}. So, we leave out the details.
	
To prove the third estimate we apply inequality \eqref{dn2.6} to obtain that 
\begin{eqnarray*}
&&\int_{I_N\times I_N}\left\vert D_m\left( x-t\right)D_n\left(
y-s\right)\right\vert d\mu \left(t\right)d\mu\left(s\right) \\
&\leq & \int_{I_N}\left\vert D_m\left( x-t\right) \right\vert d\mu
\left(t\right)\int_{I_{N}}\left\vert D_n\left( y-s\right) \right\vert d\mu\left(s\right)\leq \frac{cM_{s_1}M_{s_2}}{M_{N}^2}.
\end{eqnarray*}
	The proof is complete.
\end{proof}
	
We also consider the following maximal operators 
$ \widetilde{S}^{\ast ,p} $ and $\widetilde{S}_{\#}^{\ast}$ defined by
\begin{equation} \label{wemax}
\widetilde{S}^{\ast ,p}f =\sup_{m,n\geq 1}
\frac{\left\vert {S_{m,n}f}\right\vert}{{(m+n)}^{2/p-2}}.
\end{equation}
and
\begin{equation*}
\widetilde{S}_{\#}^{\ast}
:=\sup_{n\in \mathbb{N}}\left\vert S_{M_n,M_n}
\right\vert .
\end{equation*}

The $\sigma $-algebra generated by the 2-dimensional $I_{n}\left( x\right)\times I_{n}\left( y\right) $ square of measure $M_{n}^{-1}\times M_{n}^{-1}$
will be denoted by $\digamma _{n,n}\left( n\in \mathbb{N}\right) .$ Denote by $f=\left( f_{n,n}\text{ }n\in \mathbb{N}\right) $ one-parameter martingales with respect to $\digamma _{n,n}\left( n\in \mathbb{N}\right)$ (for details see e.g. \cite{Ne}, \cite{cw}, \cite{Webook1} and \cite{Webook2}).
	
The maximal function of a martingale $f$ \ is defined by
$
f^{\ast }=\sup_{n\in \mathbb{N}}\left\vert f_{n,n}\right\vert .
$
	
Let $f\in L_{1}\left( G_{m}^{2}\right) $. Then the maximal function is given by
\begin{equation*}
f^{\ast }\left( x,y\right) =\sup\limits_{n\in \mathbb{N}}\frac{1}{\mu \left(I_{n}(x)\times I_{n}(y)\right) }\left\vert \int\limits_{I_n(x) \times I_n(y)}f\left( s,t\right) d\mu \left( s\right)d\mu \left( t\right) \right\vert, \ \ \ \ \ \text{where} \ \ \ \left( x,y\right) \in G_{m}^{2}.
\end{equation*}

The two-dimensional Hardy space $H_{p}^{\square }(G_{m}^{2})$ $\left(
0<p<\infty \right) $ consists of all martingale for which
\begin{equation*}
\left\Vert f\right\Vert _{H_{p}^{\square}}:=\left\Vert f^{\ast }\right\Vert_{p}<\infty .
\end{equation*}
	
If $f\in L_{1}\left( G_{m}^{2}\right) ,$ then it is easy to show that the sequence $\left( S_{M_{n},M_{n}}f :n\in \mathbb{N}\right) $ is a martingale. If $f=\left( f_{n,n},n\in \mathbb{N}\right) $ is a martingale, then the Vilenkin-Fourier coefficients must be defined in a slightly different manner: $\qquad \qquad $
\begin{equation*}
\widehat{f}\left( i,j\right) :=\lim_{k\rightarrow \infty
}\int_{G_{m}^{2}}f_{k,k}\left( x,y\right) \overline{\psi }_{i}\left(
x\right) \overline{\psi }_j\left( y\right) d\mu \left( x\right)d\mu \left(y\right).
\end{equation*}

It is known (for details see e.g. Weisz \cite{Webook1}) that the following holds true for the bounded two-dimensional Vilenkin-Fourier series:

\begin{Proposition}  \label{lemma3.2.3dim2}
Let $g\in L_{1}\left( G^2_m\right)$ and $f:=(E_{n}g:n\in \mathbb{N})$ be regular martingale. Then the Vilenkin-Fourier coefficients of \ \ $f\in L_{1}\left( G_{m}^{2}\right) $ \ \
are the same as those of the martingale \\ $\left( S_{M_{n},M_{n}}f:n\in
\mathbb{N}\right) $ obtained from $f.$   

Moreover, ${H_{p}^{\square}(G_m^{2})}\text{ }(0<p\leq1)$ norm is calculated by
	\begin{equation*}
	\left\Vert f\right\Vert_{H_{p}^{\square}(G_m^{2})}=\left\Vert \sup\limits_{n\in \mathbb{N}}|S_{M_n,M_n}g|\right\Vert_{p}.
	\end{equation*}
\end{Proposition}
	
A bounded measurable function $a$ is a $p$-atom, if there exists a
\thinspace two-dimensional cube $I^{2}=I\times I\mathbf{,}$ such that
\begin{equation*}
\int_{I^{2}}ad\mu =0,\text{ \ \ \ \ \ }\left\Vert a\right\Vert _{\infty
}\leq \mu (I^{2})^{-1/p},\text{ \ \ \ \ \ supp}\left( a\right) \subset I^{2}.
\end{equation*}

\bigskip In order to prove our main results we need the following lemma of Weisz (for details see e.g. Weisz \cite{Webook1})
	
\begin{Proposition}
\label{lemma1} A martingale $f$ is in $H_{p}^{\square }\left( G_{m}^{2}\right) \left( 0<p\leq 1\right) $ if and only if there exist a sequence $\left( a_{k},k\in \mathbb{N}\right) $ of $p$-atoms and a sequence $\left( \mu _{k},k\in \mathbb{N}\right) $ of real numbers such that
\begin{equation}
\qquad \sum_{k=0}^{\infty }\mu _{k}S_{M_{n},M_{n}}a_{k}=f_{n,n},\text{ \ \ \ \ \ a.e.,}  \label{a1}
\end{equation}%
and
\begin{equation*}
\qquad \sum_{k=0}^{\infty }\left\vert \mu _{k}\right\vert ^{p}<\infty ,
\end{equation*}%
Moreover,
\begin{equation*}
\left\Vert f\right\Vert _{H_{p}^{\square}(G_m^{2})}\backsim \inf \left( \sum_{k=0}^{\infty}\left\vert \mu _{k}\right\vert ^{p}\right) ^{1/p},
\end{equation*}
where the infimum is taken over all decompositions of $f$ of the form (\ref{a1}).
\end{Proposition}

Definition several variable Hardy spaces and real Hardy spaces and related theorems of atomic decompositions of these spaces can be found in Fefferman and Stein \cite{FS} (see also Later \cite{La}, Torchinsky \cite{Tor1}, Wilson \cite{Wil}). 
	
\section{Main results}

Our first main result reads:

\begin{theorem} \label{Th1}
a) Let $0<p<1$ and $ f\in {H_{p}^{\square}(G_m^{2})} .$ Then the maximal operator $ \widetilde{S}^{\ast ,p}$ defined by \eqref{wemax} is bounded from the martingale Hardy space ${H_{p}^{\square}(G_m^{2})} $ to the space $L_{p}(G_m^{2}).$
	
b) (Sharpness) Let $0<p<1$ and $\varphi :\mathbb{N}\rightarrow \lbrack 1,\infty )$ be a non-decreasing function, satisfying the condition
$${\sup_{m,n\in\mathbb{N}}}\frac{ {(m+n)}^{2/p-2}}{\varphi \left( m,n\right) }=+\infty .  $$
Then
\begin{equation*}
\sup_{m,n\in \mathbb{N}}\left\Vert \frac{S_{m,n}f}{\varphi \left( m,n\right)}\right\Vert _{weak-L_{p}(G_m^{2})}=\infty .
\end{equation*}
\end{theorem}
	
\begin{proof}
Since $\overset{\sim }{S}_{p}^{\ast }$ is bounded from $L_{\infty }$ to $L_{\infty }$ by using Proposition  \ref{lemma1} we
conclude that the proof of part a) will be complete, if we show that
\begin{equation} \label{main}
\int\limits_{\overline{{I}\times{I}}}\left\vert \overset{\sim }{S}_{p}^{\ast
}a\left( x,y\right) \right\vert ^{p}d\mu \left( x\right)d\mu \left( y\right) \leq c<\infty,\text{ \ \ as \ \ }0<p< 1,
\end{equation}%
for every $p$-atom $a,$ where $I\times I$ denotes the support of the atom$.$
	
Let $a$ be an arbitrary $p$-atom with support $I\times I$ and $\mu \left( I\times I\right)
=M_N^{-2}.$ We may assume that $I\times I=I_{N}\times I_{N},$ where $ I_N:=I_N(0).$ It is easy to see that $S_{m,n}a=0$ when $m\leq M_N$ and $n\leq M_N.$ Therefore we can suppose that either $m>M_{N}$ or $n>M_{N}$. Since $\left\Vert a\right\Vert_{\infty}\leq M_N^{2/p}$ we find that
\begin{eqnarray} \label{two7}
\left\vert S_{m,n} a\right\vert &\leq& \int_{I_{N}\times I_{N}}\left\vert
a\left( t_1,t_2\right) \right\vert \left\vert D_{m,n}\left(x+t_1,y+t_2\right) \right\vert
d\mu \left( t_1\right)d\mu \left( t_2\right)  \\ \notag
&\leq &\left\Vert a\right\Vert _{\infty }\int_{I_{N}\times I_{N}}\left\vert D_{m,n}\left(x+t_1,y+t_2\right) \right\vert d\mu \left( t_1\right)d\mu \left( t_2\right) \\ \notag
&\leq& M_N^{2/p}\int_{I_{N}\times I_{N}}\left\vert D_{m,n}\left( x+t_1,y+t_2\right) \right\vert d\mu
\left( t_1\right)d\mu
\left( t_2\right).  \notag
\end{eqnarray}
	
Let $0<p<1$ and $(x,y)\in I_{N}\times (I_{s_2}\backslash I_{s_2+1}).$ We choose $ \varepsilon, $ so that $ 2/p-2-\varepsilon>0$  and then from Lemma \ref{dntwo2.1} and \eqref{two7} it follows that
\begin{equation}\label{13AA2}
\frac{\left\vert S_{m,n}a\left(x,y\right)\right\vert}{\left( m+n+1\right)^{2/p-2}}\leq\frac{M_N^{2/p}M_{s_2}m^{\varepsilon}} {\left(m+n+1\right)^{2/p-2}M_N^{\varepsilon+1}}\leq\frac{M_N^{2/p-1-\varepsilon}M_{s_2}} {\left(m+n+1\right)^{2/p-2-\varepsilon} }\leq M_{s_2}M_N.
\end{equation}
	
According to \eqref{2} and \eqref{13AA2} we have that
\begin{eqnarray} \label{two9}
&&\int_{{I_{N}}\times \overline{I_{N}}}\left\vert \widetilde{S}_{p}^{\ast }a \right\vert ^{p}d\mu \times d\mu 
=\overset{N-1}{\underset{s_2=0}{\sum }}\int_{I_N\times (I_{s_2}\backslash I_{s_2+1})}\left\vert \widetilde{S}_{p}^{\ast}a\right\vert ^{p}d\mu \times d\mu \leq \overset{N-1}{\underset{s_2=0}{\sum }}\frac{M^p_{s_2}}{M_{s_2}}<c_{p}<\infty . 
\end{eqnarray}
If we apply \eqref{2}, \eqref{two7} and Lemma \ref{dntwo2.1} analogously to \eqref{two9} we obtain that 
\begin{eqnarray} \label{two10}
&&\int_{\overline{{I_{N}}}\times I_{N}}\left\vert \widetilde{S} _{p}^{\ast }a \right\vert ^{p}d\mu\times d\mu=\overset{N-1} {\underset{s_1=0}{\sum }}\int_{(I_{s_1}\backslash I_{s_1+1})\times I_N}\left\vert \widetilde{S}_{p}^{\ast}a\right\vert^{p}d\mu\times d\mu \leq\overset{N-1}{\underset{s_1=0}{\sum }}\frac{M^p_{s_1}}{M_{s_1}}<c_p<\infty . 
\end{eqnarray}
	
Let $0<p<1$ and $(x,y)\in (I_{s_1}\backslash I_{s_1+1})\times (I_{s_2}\backslash I_{s_2+1}).$ By using Lemma \ref{dntwo2.1} we get that
\begin{equation}\label{13AAtwo}
\frac{\left\vert S_{m,n}a\left( x,y\right) \right\vert }{ \left( m+n+1\right)^{2/p-2}}\leq \frac{M_N^{2(1/p-1)}M_{s_1}M_{s_2}
}{\left(m+n+1\right)^{2/p-2}}\leq M_{s_1}M_{s_2}.
\end{equation}
Hence,
\begin{eqnarray} \label{two11}
&&\int_{\overline{I_N}\times \overline{I_{N}}}\left\vert \widetilde{S} _{p}^{\ast }a \right\vert ^{p}d\mu\times d\mu=\overset{N-1} {\underset{s_1=0}{\sum }}\overset{N-1} {\underset{s_2=0}{\sum }}\int_{(I_{s_1}\backslash I_{s_1+1})\times (I_{s_2}\backslash I_{s_2+1})}\left\vert \widetilde{S}_{p}^{\ast}a\right\vert^{p}d\mu\times d\mu \\ \notag &\leq&\overset{N-1}{\underset{s_1=0}{\sum }}\frac{M^p_{s_1}}{M_{s_1}}\overset{N-1}{\underset{s_2=0}{\sum }}\frac{M^p_{s_2}}{M_{s_2}}<c_p<\infty. 
\end{eqnarray}
	
Since
\begin{equation*}
\overline{{I}_{N}\times{I}_{N}}=({{I}_{N}\times\overline{{I}_{N}}})\bigcup({\overline{{I}_{N}}\times{I}_{N}})\bigcup({\overline{{I}_{N}}\times\overline{{I}_{N}}}),
\end{equation*}
by combining (\ref{two9}), (\ref{two10}) and (\ref{two11}) we get that (\ref{main}) holds for every $p$-atom and the proof of part a) is complete.
		
Now, we prove the second part of the theorem. Let $\varphi :\mathbb{N}^2\rightarrow\lbrack 1,$ $\infty)$ be a non-decreasing function and $\{\alpha_k:k\in \mathbb{N}\}$ be a sequence of natural numbers satisfying the condition
	
\begin{equation*}
\lim_{k\rightarrow\infty}\frac{\left(M_{\alpha_k}+1\right) ^{2/p-2}}{\varphi\left(2^{\alpha_k}+1,1\right)}=+\infty.
\end{equation*}
For  $ k\in \mathbb{N}_+$ set
\begin{equation*}
f_{k}\left(x,y\right)=\left(D_{M_{\alpha_k+1}}\left(x\right)-D_{M_{\alpha_k}}\left( x\right)\right)
D_{M_{\alpha_k}}\left(y\right) .
\end{equation*}
It is evident that
\begin{equation*}
\widehat{f}_k\left(i,j\right) =
\begin{cases}
1,&\text{ if }(i,j)\in \{ \underset{l=0}{\overset{k}{\mathop{\cup }}}\{M_{\alpha _l},...,M_{\alpha_l+1}-1\}\}\times \{0,...,M_{\alpha_k+1}-1\}, \\
0,&\text{ otherwise}. 
\end{cases}
\end{equation*} 
Therefore, 
\begin{equation} \label{two14a}
S_{i,j}(f_k; x,y)=
\end{equation}
\begin{equation*}
\begin{cases}
\left(D_i\left(x\right)-D_{M_{\alpha_k}}\left(x\right)
\right) D_j\left( y\right),&\text{ if }(i,j)\in \{ \underset{l=0}{\overset{k}{\mathop{\cup }}}\{M_{\alpha _l},...,M_{\alpha_l+1}-1\}\}\times \{1,...,M_{\alpha_k+1}-1\},\\
f_k\left( x,y\right) &\text{ if }i\geq M_{\alpha_k+1} {\ \ \text {and} \ \ } j\geq M_{\alpha_k+1}, \\
0,&\text{ otherwise}. 
\end{cases}
\end{equation*}
From \eqref{two14a} it follows that 
\begin{eqnarray} \label{main1}
\left\Vert f_{{k}}\right\Vert _{H_{p}^{\square }}
&=&\left\Vert \sup\limits_{n\in
\mathbb{N}}S_{M_n,M_n} f_{k} \right\Vert_{p} =\left\Vert
\left(D_{M_{\alpha_k+1}}(x)-D_{M_{\alpha _{k}}}(x)\right)D_{M_{\alpha_k}}(y)\right\Vert_p 
\leq c_p M_{\alpha_k}^{2\alpha_{k}\left(1-1/p\right)}.\\ \notag
	\end{eqnarray}
	Let $(x,y)\in G_m^2$. Moreover, \eqref{two14a} also implies that 
\begin{eqnarray*}
\frac{\left\vert S_{M_{\alpha _{k}}+1,1}(f_{{k}};x,y) \right\vert }{\varphi \left( M_{\alpha _{k}}+1,1\right) } 
=\frac{\left\vert \left( D_{M_{\alpha_{k}}+1}\left( x\right)
-D_{M_{\alpha _{k}}}\left( x\right) \right) D_{1}\left( y\right) 
\right\vert}{\varphi \left(M_{\alpha _{k}}+1,1\right)} =\frac{\left\vert w_{M_{\alpha _{k}}}\left( x\right) w_{0}\left( y\right) \right\vert }{\varphi \left( M_{\alpha
_{k}}+1,1\right)}=\frac{1}{\varphi \left(M_{\alpha _{k}}+1,1\right)}.
	\end{eqnarray*}
	
Hence, by also using \eqref{main1}, we find that
\[\begin{split}
&\frac{\frac{1}{\varphi \left( M_{\alpha_k}+1,1\right)}\left( \mu \left\{ (x,y)\in G^2_m:\frac{
S_{M_{\alpha _{k}}+1,1}(f_k;x,y)  }{\varphi
\left( M_{\alpha _{k}}+1,1\right) }\geq \frac{1}{\varphi \left( M_{\alpha_{k}}+1,1\right) }\right\} \right)^{1/p}}
{\left\Vert f_{{k}}\right\Vert _{H_p^{\square}(G_m^2)}} \\
&\geq \frac{1}{\varphi \left( M_{\alpha _{k}}+1,1\right)M_{\alpha_k}^{2\left( 1-1/p\right) }}\geq \frac{\left( M_{\alpha
_{k}}+1\right)^{2/p-2}}{\varphi \left( M_{\alpha _{k}}+1,1\right) }
\rightarrow \infty,  \ \ \ \text{ as } \ \ \ k\rightarrow \infty .
\end{split}\]

The proof is complete.
\end{proof}

We also apply Theorem \ref{Th1} to obtain that the following is true:

\begin{theorem} \label{Th2}
	Let $0<p<1,$ $f\in H_{p}^{\square}(G_m^2),$  $2^{-\alpha}<m/n\leq 2^{\alpha}$ and $ 2^{k}<m,n\leq 2^{k+1+[\alpha]}. $ Then there exists an absolute constant $ c_p, $ such that
	\begin{equation*}
	\left\Vert S_{m,n}f-f\right\Vert_{{H_{p}^{\square}(G_m^{2})} }\leq c_{p}{M}_{k}^{2/p-2}\omega_{H_{p}^{\square}(G_m^{2})} \left( \frac{1}{M_{k}},f\right).  
	\end{equation*}
\end{theorem}

\begin{proof} Without lost a generality we may assume that $n<m.$ According to Theorem \ref{Th1} we can conclude that
\begin{equation*}
\left\Vert S_{m,n}f\right\Vert _{p}\leq c_{p}^1{(m+n)}^{2/p-2}\left\Vert f\right\Vert _{H_p^{\square}(G_m^2)}\leq c_{p}^2 M_{k+1+[\alpha]}^{2/p-2}\left\Vert f\right\Vert_{H_p^{\square}(G_m^2)}\leq c_p^3 M_k^{2/p-2}\left\Vert f\right\Vert_{H_p^{\square}(G_m^{2})}.
\end{equation*}

Since $ 2^{-\alpha }\leq m/n\leq 2^{\alpha }$ we obtain that 
$$M_{k}<m,n,M_{|n|},M_{|n|+1}...,M_{|m|}\leq M_{k+1+[\alpha]}, \ \ {(|m|-|n|+1)}\leq\alpha+2$$
and
\begin{equation*}
\left\Vert S_{M_{i},n}f\right\Vert _{p}\leq c_p^3 M_k^{2/p-2}\left\Vert f\right\Vert_{H_p^{\square}(G_m^{2})},\ \ \text{where } \ \ |n|\leq i \leq |m|.
\end{equation*} 

Let us consider the following martingale $f_{\#}:=\left(S_{2^k,2^k}S_{m,n}f,\text{ }k\in\mathbb{N}_+\right).$ By a simple calculation we get that
\begin{equation*}
f_{\#}
=\left(S_{M_0,M_0}f,S_{M_{|n|},M_{|n|}}f,S_{M_{|n|+1},n}f,...,
S_{M_{|m|},n}f, S_{m,n}f,...,S_{m,n}f,...\right).
\end{equation*}

By using Proposition \ref{lemma3.2.3dim2} we immediately get that
\begin{eqnarray*}
\left\Vert S_{m,n}f\right\Vert _{H_p^{\square}(G_m^2)}^p
&\leq & \left\Vert \sup_{0\leq l\leq |n|}\left\vert S_{M_l,M_l}f\right\vert \right\Vert_p^p+\overset{|m|}{\underset{i=|n|}{\sum }}\left\Vert
S_{M_i,n}f\right\Vert_p^p
+\left\Vert S_{m,n}f\right\Vert_p^p\\
&\leq &\left\Vert \widetilde{S}_{\#}^{\ast } f\right\Vert_p^p+
{(|m|-|n|+1)}M_k^{2/p-2}
\left\Vert f\right\Vert_{H_p^{\square}(G_m^2)}^p \\
&\leq &c_p^5\left\Vert f\right\Vert_{H_p^{\square}(G_m^2)}^p+
{c_p^6}M_k^{2/p-2}
\left\Vert f\right\Vert_{H_p^{\square}(G_m^2)}^p\leq {c_p}M_k^{2/p-2}
\left\Vert f\right\Vert_{H_p^{\square}(G_m^2)}^p.
\end{eqnarray*}

Hence,
\begin{eqnarray*}
\left\Vert S_{m,n}f-f\right\Vert _{H_p^{\square}(G_m^2)}^p &\leq&  \left\Vert
S_{m,n}f-S_{M_k,M_k}f\right\Vert _{H_p^{\square}(G_m^2)}^p+\left\Vert S_{M_{k},M_{k}}f-f\right\Vert_{H_p^{\square}(G_m^2)}^p \\
&=&\left\Vert S_{m,n}\left( S_{M_{k},M_{k}}f-f\right) \right\Vert _{H_p^{\square}(G_m^2)}^p+\left\Vert S_{M_k,M_k}f-f\right\Vert _{H_p^{\square}(G_m^2)}^p \\
&\leq & c_{p}^2\left(M_k^{2-2p}+1\right)\omega_ {H_p^{\square}(G_m^2)}^p\left( \frac{1}{M_{k}},f\right)
\end{eqnarray*}%
and
\begin{equation*}
\left\Vert S_{m,n}f-f\right\Vert _{H_p^{\square}(G_m^2)}^p\leq c_{p}{M}_k^{2/p-2}\omega_{H_{p}^{\square}(G_m^{2})} \left( \frac{1}{M_{k}},f\right).  \label{app}
\end{equation*}
	
The proof is complete.
\end{proof}

\begin{theorem}
a) Let $0<p<1,$ $f\in H_{p}^{\square}(G_m^{2}),$ $2^{-\alpha }\leq m/n\leq 2^{\alpha }$ and
\begin{equation*} \label{10A}
\omega_{H_{p}^{\square}(G_m^{2})} \left( \frac{1}{M_{k}},f\right)=o\left( \frac{1}{M_{k}^{2/p-2}}\right),\text{ as \ }k\rightarrow \infty .  
\end{equation*}
Then
\begin{equation*}
\left\Vert S_{m,n}f-f\right\Vert_{H_{p}^{\square}(G_m^{2})}\rightarrow
0,\text{ as }m,n\rightarrow \infty .
\end{equation*}
	
b) (Sharpness) Let $0<p<1$ and $2^{-\alpha}<m/n\leq 2^{\alpha}.$ Then there exists a martingale $f\in H_{p}^{\square}(G_m^{2}),$ such that 
\begin{equation*}
\omega_{H_{p}^{\square}(G_m^{2})} \left( \frac{1}{M_{k}},f\right)=O\left( \frac{1}{M_k^{2/p-2}}\right),\text{ as \ }k\rightarrow \infty
\end{equation*}
and
\begin{equation*}
\left\Vert S_{m,n}f-f\right\Vert _{weak-L_p(G_m^{2})}\nrightarrow
0\text{ as  }m,n\rightarrow \infty .
\end{equation*}
\end{theorem}	
		
\begin{proof} Let $0<p<1,$ $f\in {H_{p}^{\square}(G_m^{2})},$ $ 2^{-\alpha }\leq m/n\leq 2^{\alpha }$  and
\begin{equation*}
\omega_{H_{p}^{\square }(G_m^{2})} \left( \frac{1}{M_{k}},f\right)=o\left( \frac{1}{
M_k^{2/p-2}}\right) ,\text{ \ \ \ as \ \ \ }k\rightarrow \infty .
\end{equation*}%
By using Theorem \ref{Th2} we immediately get that
\begin{equation*}
\left\Vert S_{m,n}f-f\right\Vert_{H_p^{\square}(G_m^2)}\rightarrow \infty ,\ \ \ \text{ as }\ \ \
n\rightarrow \infty 
\end{equation*}
	and the proof of part a) is complete.
	
	Let \qquad 
	\begin{equation*}
	f_{n,n}=\sum_{\left\{ k;\text{ }\alpha _{k}+1<n\right\} }\lambda _{k}a_{k},
	\end{equation*}%
	where 
	\begin{equation*}
	\lambda _{k}=M_{\alpha _{k}}^{-\left( 2/p-2\right)}
	\end{equation*}%
	and
	
\begin{equation*}
a_{k}\left( x,y\right)=M_{\alpha _{k}}^{2/p-2}\left(
D_{M_{\alpha _{k}+1}}\left( x\right) -D_{M_{\alpha _{k}}}\left( x\right)
\right) \left(
D_{M_{\alpha_{k}+1}}\left( y\right)-D_{M_{\alpha_k}}\left( y\right)
\right).
\end{equation*}
	
Since
\begin{equation*}
S_{M_{n},M_{n}}a_{k}=\left\{ 
\begin{array}{l}
a_{k}\text{, \qquad }\alpha _{k}+1<n, \\ 
0\text{, \qquad\ \ }\alpha _{k}+1\geq n,
\end{array}\right. 
\end{equation*}

\begin{equation*}
\text{supp}(a_{k})=I_{\alpha _{k}}^{2},\text{ \qquad }\int_{I_{\alpha
_{k}}^{2}}a_{k}d\mu =0,\text{ \qquad }\left\Vert a_{k}\right\Vert _{\infty}\leq M_{\alpha _{k}}^{2/p}=(\text{supp }a_{k})^{-1/p}
\end{equation*}%
from Lemma \ref{lemma1} and the fact that
$\sum_{k=0}^{\infty }\left| \mu _{k}\right| ^{p}<\infty ,$
we conclude that $f\in H_{p}^{\square }(G_m^{2}).$
	
	Moreover, for all $k\in \mathbb{N}_{+},$
\begin{eqnarray}\label{20two20}
&&f-S_{M_{n},M_{n}}f \\ 
&=&\left( f^{\left( 1\right) }-S_{M_{n},M_{n}}f^{\left( 1\right) },...,f^{\left(n\right) }-S_{M_{n},M_{n}}f^{\left( n\right) },...,f^{\left( n+k\right)} -S_{M_{n},M_{n}}f^{\left( n+k\right) }\right) \notag \\
&=&\left( 0,...,0,f^{\left( n+1\right) }-f^{\left( n\right) },...,f^{\left(n+k\right) }-f^{\left( n\right) },...\right)=\left( 0,...,0,\underset{i=n}{\overset{n+k}{\sum }}\frac{a_{i}(x,y)}{
M_i^{2/p-2}},...\right)\notag
\end{eqnarray}
is a martingale and \eqref{20two20} is its atomic decomposition. By using Lemma \ref{lemma1} we find that
\begin{equation*}
\omega _{H_{p}^{\square }(G_m^{2})}\left( \frac{1}{M_n},f\right) :=\left\Vert f-S_{M_n,M_n}f\right\Vert_{H_{p}^{\square }(G_m^{2})}\leq \underset{i=n}{\overset{\infty}{\sum }}\frac{1}{
M_i^{2/p-2}}\leq \frac{c}{M_n^{2/p-2}}.
\end{equation*}
	
It is easy to show that 
\begin{eqnarray} \label{5}
&&\widehat{f}(i,j)   =\left\{ 
\begin{array}{l}
1,\text{ \ if \ }\left( i,j\right) \in \left\{
M_{\alpha _{k}},...,M_{\alpha _{k}+1}-1\right\}^2 ,\text{ }k\in \mathbb{N}, \\ 
0,\text{ \ if \ }\left( i,j\right) \notin \bigcup\limits_{k=1}^{\infty
}\left\{M_{\alpha_{k}},...,M_{\alpha _{k}+1}-1\right\}^2.
\end{array}%
\right.   \notag
\end{eqnarray}
	
Hence,
\begin{eqnarray} \label{13d}
S_{M_{\alpha _{k}+1},M_{\alpha _{k}+1}}f\left( x,y\right)= S_{M_{\alpha _{k}},M_{\alpha _{k}}}f\left( x,y\right)+w_{M_{\alpha _{k}}}\left( x\right)w_{M_{\alpha _{k}}}\left( y\right)=:I+II. 
\end{eqnarray}
	
It is obvious that
$\left\vert II\right\vert=\left\vert w_{M_{2\alpha_k}}\left( x\right)w_{M_{\alpha _{k}}}\left( y\right)\right\vert= 1$
and
\begin{eqnarray}  \label{13}
&&\left\Vert II\right\Vert _{weak-L_{p}(G_m^{2})}^p\geq \frac{1}{2^p}\left( \mu \left\{ \left( x,y\right)\in G_m^2
:\left\vert II \right\vert \geq \frac{1}{2}\right\} \right)\geq \frac{1}{2^p}\mu \left(G_m^2 \right) \geq \frac{1}{2^p}. 
\end{eqnarray}
	
Since (for details see e.g. Weisz \cite{Webook1} and \cite{Webook2})
\begin{eqnarray*}
\Vert f-S_{M_{n},M_{n}}f\Vert _{weak-L_{p}(G_m^{2})}\rightarrow 0,  \ \ \text{as} \ \ n\rightarrow \infty.
\end{eqnarray*}
	
According to (\ref{13d}) and (\ref{13}) we obtain that
\begin{eqnarray*} \label{144}
&&\limsup\limits_{k\rightarrow \infty }\Vert f-S_{M_{\alpha _{k}}+1,M_{\alpha _{k}}+1}f\Vert _{weak-L_{p}(G_m^{2})}^{p} \\ \notag
&\geq&\limsup\limits_{k\rightarrow \infty }\Vert II\Vert _{weak-L_{p}(G_m^{2})}^p-\limsup\limits_{k\rightarrow \infty }\Vert f-S_{M_{\alpha _{k}},M_{\alpha _{k}}}f\Vert _{weak-L_{p}(G_m^{2})}^p \geq c>0.
\end{eqnarray*}
	
The proof is complete.
\end{proof}

\begin{theorem}
a) Let $0<p<1, \ \ f\in H_p^{\square}\left( G^2_m\right).$ Then there exists an absolute constant $c_p, $ depending only of $p, $ such that 
\begin{equation*}
\underset{{ \{\left(k,l\right):\ 2^{-\alpha }\leq k/l\leq 2^{\alpha }\}}}{\sum }\frac{\left\Vert
S_{k,l}f\right\Vert _{p}^{p}}{\left( kl\right) ^{2-p}}\leq c_{p}\left\Vert f\right\Vert_{H_{p}^{\square}(G_m^{2})}^{p}.
\end{equation*}
	
b) Let $0<p<1$ and $\Phi :\mathbb{N}^2\rightarrow \lbrack 1, \infty )$ be non-decreasing, nonnegative function, satisfying the condition 
\begin{equation} \label{dir222}
\overline{\underset{m,n\rightarrow \infty }{\lim }}\Phi \left( m,n\right)=+\infty .  
\end{equation}
	
Then there exists a martingale $f\in {H_{p}^{\square}(G_m^{2})} $ such that		
\begin{equation*}
\underset{ \{\left(k,l\right):\ \ 2^{-\alpha }\leq k/l\leq 2^{\alpha }\}}{\sum}\frac{ \left\Vert S_{k,l}f\right\Vert_p^p\Phi \left( k,l\right)}{(kl)^{2-p}}=\infty.
\end{equation*}
\end{theorem}
	
\begin{proof} If we follow similar steps of the proof of Theorem \ref{Th1} then we obtain that either $m>M_{N}$ or $n>M_{N}$, but in this case under condition $ 2^{-\alpha }\leq m/n\leq 2^{\alpha } $ we can conclude that $M_{N-1-[\alpha]}<m,n\leq M_{N+1+[\alpha]}.$
If we apply Proposition \ref{lemma1} we only have to prove that 
\begin{eqnarray*}
&& \underset{2^{-\alpha }\leq k/l\leq 2^{\alpha }, \ \ 
k,l> M_{N-1-[\alpha]}}{\sum}\frac{\left\Vert S_{k,l}a\right\Vert_ {p}^{p}}{(kl)^{2-p}}= \underset{2^{-\alpha }\leq k/l\leq 2^{\alpha }, \ \ k,l>M_{N-1-[\alpha]}}{\sum}\left\Vert\frac{ S_{k,l}a}{(kl)^{2/p-1}}\right\Vert_p^p<c_p<\infty.
\end{eqnarray*}

Let $0<p<1$ and $(x,y)\in I_{N}\times (I_{s_2}\backslash I_{s_2+1}).$
 We choose $ \varepsilon, $ so that $ 2/p-1-\varepsilon >0.$ Since
\begin{equation}\label{13AA2aaaa}
\frac{\left\vert S_{k,l}a\left( x,y\right) \right\vert }{\left( kl\right)^{2/p-1}}
\leq \frac{M_N^{2/p}M_{s_2}k^{\varepsilon}} {\left(kl\right)^{2/p-1}M_N^{\varepsilon+1}}
\leq\frac{M_{s_2}M^{2/p-1-\varepsilon}_N}{k^{2/p-1-\varepsilon}l^{2/p-1}}
\end{equation}
according to \eqref{2} and \eqref{13AA2aaaa} we get that
\begin{eqnarray} \label{two9aaaa}
&&\int_{{I_{N}}\times \overline{I_{N}}}\left\vert \frac{ S_{k,l}a}{(kl)^{2/p-1}}\right\vert ^{p}d\mu \times d\mu \\ \notag
&=&\overset{N-1}{\underset{s_2=0}{\sum }}\int_{I_N\times (I_{s_2}\backslash I_{s_2+1})}\left\vert \frac{ S_{k,l}a}{(kl)^{2/p-1}}\right\vert ^{p}d\mu \times d\mu \leq\frac{M^{1-p-p\varepsilon}_N}{k^{2-p-p\varepsilon}l^{2-p}}
\overset{N-1}{\underset{s_2=0}{\sum}}\frac{M^p_{s_2}}{M_{s_2}}<\frac{c_p M^{1-p-p\varepsilon}_N}{k^{2-p-p\varepsilon}l^{2-p}}.  \notag
\end{eqnarray}

By using  \eqref{two9aaaa} we find that
\begin{eqnarray*}
&&\underset{2^{-\alpha }\leq k/l\leq 2^{\alpha }, \ \ 
k,l\geq M_{N-1-[\alpha]}}{\sum}\int_{{I_{N}}\times \overline{I_{N}}}\left\vert \frac{ S_{k,l}a}{(kl)^{2/p-1}}\right\vert ^{p}d\mu \times d\mu \\
&\leq&c_p M_N^{1-p-p\varepsilon}
\sum_{k=M_{N-1-[\alpha]}}^{\infty}\frac{1}{k^{2-p-p\varepsilon}}\sum_{l=M_{N-1-[\alpha]}}^{\infty}\frac{1}{l^{2-p}}\\
&\leq&c_p M_N^{1-p-p\varepsilon}\frac{1}{M_{N-1-[\alpha]}^{1-p-p\varepsilon}}\frac{1}{M_{N-1-[\alpha]}^{1-p}}\leq c_p M_N^{2-p-p\varepsilon} \frac{1}{M_{N-1-[\alpha]}^{3-2p-p\varepsilon}}<\frac{c_p}{M_N^{1-p}}<c_p<\infty.
\end{eqnarray*}

Let $0<p<1$ and $(x,y)\in (I_{s_1}\backslash I_{s_1+1})\times I_{N}$ and $ \varepsilon$ be real number, so that $ 2/p-1-\varepsilon >0.$ Analogously, we can prove that
\begin{eqnarray} \label{9aaaa}
&&\int_{\overline{I_{N}}\times {I_{N}}}\left\vert \frac{ S_{k,l}a}{(kl)^{2/p-1}}\right\vert ^{p}d\mu \times d\mu  \leq\frac{M^{1-p-p\varepsilon}_N}{k^{2-p}l^{2-p-p\varepsilon}}
\overset{N-1}{\underset{s_1=0}{\sum}}\frac{M^p_{s_1}}{M_{s_1}}<\frac{c_p M^{1-p-p\varepsilon}_N}{k^{2-p}l^{2-p-p\varepsilon}}.
\end{eqnarray}

It follows that
\begin{eqnarray*}
&&\underset{2^{-\alpha }\leq k/l\leq 2^{\alpha },  \ 
k,l\geq M_{N-1-[\alpha]}}{\sum}\int_{ \overline{I_{N}}\times {I_{N}}}\left\vert \frac{ S_{k,l}a}{(kl)^{2/p-1}}\right\vert^p d\mu \times d\mu \\
&\leq&c_p M_N^{1-p-p\varepsilon}
\sum_{k=M_{N-1-[\alpha]}}^{\infty}\frac{1}{k^{2-p}}\sum_{l=M_{N-1-[\alpha]}}^{\infty}\frac{1}{l^{2-p-p\varepsilon}}<\frac{c_p}{M_N^{1-p}}<c_p<\infty.
\end{eqnarray*}
Let $0<p<1$ and $(x,y)\in (I_{s_1}\backslash I_{s_1+1})\times (I_{s_2}\backslash I_{s_2+1}).$ Then by using Lemma \ref{dntwo2.1} we get that
\begin{equation}\label{13two}
\left (\frac{\left\vert S_{m,n}a\left( x,y\right) \right\vert }{ \left( mn\right)^{2/p-1}}\right )^p\leq \frac{M_N^{2-2p}M^p_{s_1}M^p_{s_2}
}{\left(mn\right)^{2-p}}.
\end{equation}
In view of (\ref{2}) and (\ref{13two}) we can conclude that
\begin{eqnarray} \label{12}
&&\int_{\overline{I_{N}}\times \overline{I_{N}}}\left (\frac{\left\vert S_{m,n}a \right\vert }{ \left( mn\right)^{2/p-1}}\right )^p d\mu \times d\mu =\frac{M_N^{2-2p}
}{\left(mn\right)^{2-p}}\overset{N-1}{\underset{s_1=0}{\sum }}\overset{N-1}{\underset{s_2=0}{\sum }} \int_{({I_{s_1}\backslash I_{s_1+1}})\times (I_{s_2}\backslash I_{s_2+1})}\left\vert \widetilde{S}_{p}^{\ast
}a\right\vert^{p}d\mu \times d\mu \\ \notag
&\leq &\frac{M_N^{2-2p}
}{\left(mn\right)^{2-p}} \overset{N-1}{\underset{s_1=0}{\sum }}\overset{N-1} {\underset{s_2=0}{\sum }}\frac{M_{s_1}^pM_{s_2}^p}{M_{s_1}M_{s_2}}\leq \frac{M_N^{2-2p}
}{\left(mn\right)^{2-p}}.  
\end{eqnarray}
Hence,
\begin{eqnarray*}
&&\underset{2^{-\alpha }\leq k/l\leq 2^{\alpha }, \ \ k,l\geq M_{{N-1-[\alpha]}}}{\sum}\int_{\overline{I_{N}}\times \overline{I_{N}}}\left\vert \frac{ S_{k,l}a}{(kl)^{2/p-1}}\right\vert^p d\mu \times d\mu \\
&\leq&c_p M_N^{2-2p}\sum_{k=M_{N-1-[\alpha]}}^{\infty}\frac{1} {k^{2-p}} \sum_{l=M_{N-1-[\alpha]}}^{\infty}\frac{1}{l^{2-p}}\leq c_pM_N^{2-2p} \frac{1}{M_{N-1-[\alpha]}^{1-p}}\frac{1}{M_{N-1-[\alpha]}^{1-p}}\leq c_p<\infty.
\end{eqnarray*}

The proof of part a) is complete.
		
Under the condition (\ref{dir222}) there exists an
increasing sequence of positive integers $\left\{ \alpha _{k}:\text{ }%
k\geq 0\right\} $ such that $\alpha _{0}\geq 2,$ \ $\alpha_k+[\alpha]+1<\alpha_{k+1}$ \ and 
\begin{equation}\label{2aaaa}
\sum_{k=0}^{\infty}{\Phi^{-p/4}\left(M_{\alpha _{k}},M_{\alpha_k}\right)}<\infty.  
\end{equation}

Let 
\begin{equation*}
f_{n,n}=\sum_{\left\{ k;\text{ }\alpha _{k}+[\alpha]+1<n\right\} }\lambda
_{k}a_{k},
\end{equation*}
where 
$\lambda _{k}:=\left(m_{\alpha_k}...m_{\alpha_k+[\alpha]}\right)^{2/p-2}{\Phi ^{-1/4}\left(M_{\alpha _k},M_{\alpha_k}\right)}$
and
\begin{equation*}
a_{k}\left(x,y\right):=M_{\alpha_{k}+[\alpha]+1}^{2/p-2}\left(
D_{M_{\alpha_k+[\alpha]+1}}\left(x\right)-D_{M_{\alpha _{k}}}\left(
x\right) \right) \left( D_{M_{\alpha _{k}+[\alpha]+1}}\left( y\right)
-D_{M_{\alpha _{k}}}\left( y\right) \right).
\end{equation*}

Since
\begin{equation*}
S_{2^{n},2^{n}}a_{k}=\left\{ 
\begin{array}{l}
a_{k}\text{, \qquad } \alpha _{k}+[\alpha]+1<n, \\ 
0\text{, \qquad } \alpha _{k}+[\alpha]+1\geq n,%
\end{array}
\right.
\end{equation*}
\begin{eqnarray*}
&&\text{supp}(a_{k})=I_{\alpha _{k}}^2, \text{ \qquad } \ \ \ \int_{I_{\alpha_{k}}^2}a_k d\mu=0, \text{ \qquad } \ \ \ \left\| a_{k}\right\| _{\infty }\leq M_{\alpha_k}^{2/p} =(\mu(\text{supp}a_{k}))^{-1/p}
\end{eqnarray*}
from Proposition 1 and (\ref{2aaaa}) we obtain that $f\in H_{p}^{\square }(G_m^{2}).$ It is obvious that 

\begin{eqnarray}  \label{5aaaa}
\widehat{f}(i,j)=\left\{ 
\begin{array}{l}
\frac{M_{\alpha _k}^{2/p-2}}{\Phi^{1/4}\left( M_{\alpha_k},M_{\alpha_k}\right)}, \text{ \ if \ }\left( i,j\right) \in
\left\{M_{\alpha_k},...,M_{\alpha_k+[\alpha]+1}-1\right\}^2,\text{ \ \ \ }i,j\in \mathbb{N} \\ 
0,\text{ \ if \ }\left(i,j\right) \notin \bigcup\limits_{k=1}^{\infty
}\left\{M_{\alpha_k},...,M_{\alpha_k+[\alpha]+1}-1\right\}^2.
\end{array}
\right.  \notag
\end{eqnarray}

Let\textbf{\ } $M_{\alpha_k}<m,n<M_{\alpha_k+[\alpha]+1}$. Then

\begin{eqnarray}\label{13daaaa}
&&S_{m,n}f\left( x,y\right)=\sum_{\eta =0}^{k-1}\sum_{i=M_{\alpha _{\eta }}}^{M_{\alpha _{\eta
}+1}-1}\sum_{j=M_{\alpha_{\eta }}}^{M_{\alpha _{\eta }+1}-1}\widehat{f}
(i,j)w_{i}\left( x\right) w_{j}\left( y\right)+\sum_{i=M_{\alpha _{k}}}^{m-1}\sum_{j=M_{\alpha _{k}}}^{n-1}\widehat{f}
(i,j)w_{i}\left( x\right) w_{j}\left( y\right)  
\end{eqnarray}
\begin{eqnarray*}
&=&\sum_{\eta =0}^{k-1}\sum_{i=M_{\alpha_{\eta }}}^{M_{\alpha _{\eta
}+1}+[\alpha]-1}\sum_{j=M_{\alpha_{\eta}}}^{M_{\alpha _{\eta }+1}+[\alpha]-1}\frac{M^{2/p-2}_{\alpha_{\eta}}w_{i}\left( x\right) w_{j}\left( y\right)}{\Phi ^{1/4}\left(M_{\alpha _{\eta}},M_{\alpha _{\eta }}\right) }+\sum_{i=M_{\alpha_k}}^{m-1}\sum_{j=2^{\alpha _{k}}}^{n-1}\frac{M^{2/p-2}_{\alpha_k}w_i\left( x\right) w_j\left( y\right)}{\Phi^{1/4}\left(M_{\alpha _{k}},M_{\alpha _{k}}\right)}  \notag \\
&=&\sum_{\eta =0}^{k-1}\frac{M^{2/p-2}_{\alpha_{\eta }}}{\Phi^{1/4}\left(M_{\alpha_{\eta}},M_{\alpha _{\eta }}\right) }\left(
D_{M_{\alpha _{\eta }+[\alpha]+1}}\left( x\right) -D_{M_{\alpha _{\eta }}}\left(
x\right) \right) \left( D_{M_{\alpha_{\eta }+[\alpha]+1}}\left( y\right)
-D_{M_{\alpha _{\eta }}}\left( y\right) \right)  \notag \\
&&+\frac{M^{2/p-2}_{\alpha_k}}{\Phi ^{1/4}\left(M_{\alpha
_{k}},M_{\alpha _{k}}\right) }\left( D_{m}\left( x\right) -D_{M_{\alpha_k}}\left( x\right) \right) \left( D_{n}\left( y\right) -D_{M_{\alpha_{k}}}\left( y\right) \right):=I+II.  \notag
\end{eqnarray*}

Let $\left( x,y\right) \in \left( G\backslash I_{1}\right) \times \left(G\backslash I_{1}\right),$ $m,n\in \mathbb{N}_{n_{0}},$ such that $M_{\alpha _{k}}<m,n<2^{\alpha}M_{\alpha _{k}}$. Since  $\alpha _{k}\geq 2$ $\left( k\in \mathbb{N}\right),$ if we combine (\ref{1dn})-(\ref{2dn}) it follows that
$D_{M_{\alpha_k}}\left( x\right)=D_{M_{\alpha_k}}\left(y\right)=0$
and
\begin{eqnarray}  \label{13aaaaatwo}
\left\vert II\right\vert=\frac{M_{\alpha_k}^{2/p-2}}{\Phi ^{1/4}\left({M_{\alpha_k}},{M_{\alpha_k}}\right)}\left\vert w_m\left( x\right)w_1\left(x\right)D_1\left(x\right) w_n \left( y\right)D_1\left( y\right)\right\vert=\frac{M_{\alpha_k}^{2/p-2}}{\Phi ^{1/4}\left(M_{\alpha_k},M_{\alpha_k}\right)}. 
\end{eqnarray}
By applying (\ref{1dn}) and the condition $\alpha _{n}\geq 2$ $\left( n\in 
\mathbb{N}\right) $ for $I$ we have that

\begin{equation} \label{13btwo}
I=\sum_{\eta =0}^{k-1}\frac{M_{\alpha _{\eta} }^{2/p-2}\left( D_{M_{\alpha _{\eta }+1+1}}\left( x\right) -D_{M_{\alpha _{\eta }+1}}\left( x\right) \right) \left(D_{M_{\alpha _{\eta }+1}}\left(y\right)-D_{M_{\alpha _{\eta }+1}}\left(
	y\right) \right) }{\Phi
^{1/4}\left( M_{\alpha _{\eta }+1,M_{\alpha _{\eta }+1}}\right)}=0.  
\end{equation}
By combining (\ref{13aaaaatwo}) and (\ref{13btwo}) for $M_{\alpha _{k}}<m,n<2^{\alpha}M_{\alpha _{k}}$ we get that

\begin{eqnarray} \label{13aaaabbtwo}
&&\left\Vert S_{m,n}f \right\Vert _{weak-L_p(G_m^2)}  
\end{eqnarray}
\begin{eqnarray}
&\geq &\frac{M_{\alpha_{k}}^{ 2/p-2}}{2\Phi ^{1/4}\left(
M_{\alpha _{k}},M_{\alpha _{k}}\right) }\left( \mu \left\{ \left( x,y\right)\in \left( G\backslash I_{1}\right) \times \left( G\backslash I_{1}\right):\left\vert S_{m,n}f\left( x,y\right) \right\vert \geq \frac{M_{\alpha_k}^{2/p-2}}{2\Phi ^{1/4}\left( M_{\alpha _{k}},M_{\alpha_{k}}\right) }\right\} \right) ^{1/p}  \notag \\
&\geq &\frac{M_{\alpha_k}^{2/p-2} }{2\Phi ^{1/4}\left(
M_{\alpha _{k}},M_{\alpha _{k}}\right) }\mu \left( \left( G_m\backslash
I_{1}\right)\right) \geq \frac{c_{p}M_{\alpha _{k}}^{2/p-2 }}{\Phi ^{1/4}\left( M_{\alpha_{k}},M_{\alpha _{k}}\right) }.  \notag
\end{eqnarray}

According to (\ref{1atwo}), (\ref{dir222}) and (\ref{13aaaaatwo})-(\ref{13aaaabbtwo}) we can conclude
that

\begin{eqnarray} \label{14}
&&\underset{{ \{\left(k,l\right):\ \ 2^{-\alpha }\leq k/l\leq 2^{\alpha }\}}}{\sum }\frac{\left\Vert S_{m,n}f\right\Vert _{weak-L_{p}}^{p}\Phi \left( m,n\right) }{\left( mn\right)^{2-p}} 
\geq \underset{M_{\alpha_k}<m,n\leq 2^{\alpha}M_{\alpha _{k} }}{\sum }\frac{\left\Vert S_{m,n}f\right\Vert _{weak-L_{p}}^{p}\Phi \left( m,n\right)}{\left( mn\right)^{2-p}} \\ \notag
\end{eqnarray}
\begin{eqnarray}
&\geq &\frac{c_{p}\Phi \left(M_{\alpha _{k}},M_{\alpha _{k}}\right) }{
M_{\alpha_{k}}^{4-2p}}\underset{M_{\alpha _{k}}<m,n\leq
2^{\alpha}M_{\alpha _{k}}}{\sum }\left\Vert S_{m,n}f\right\Vert
_{weak-L_{p}}^{p} \\ \notag
&\geq &\frac{c_{p}\Phi \left(M_{\alpha_k},M_{\alpha _{k}}\right) }{
M_{\alpha_{k}}^{4-2p}}\underset{M_{\alpha_{k}-1}<m,n\leq
2^{\alpha}M_{\alpha_k-1}, \ n\in \mathbb{N}_{n_{0}}}{\sum }\left\Vert S_{m,n}f\right\Vert
_{weak-L_{p}}^{p} \\ \notag
&\geq &\frac{c_{p}\Phi \left(M_{\alpha _{k}},M_{\alpha _{k}}\right) }{
M_{\alpha _{k}}^{ 4-2p }}\frac{M_{\alpha _{k} }^{2-2p}}{\Phi ^{1/4}\left(M_{\alpha _{k}},M_{\alpha _{k}}\right) }\underset{M_{\alpha_{k}-1}<m,n\leq 2^{\alpha}M_{\alpha _{k}-1}, \  n\in \mathbb{N}_{n_{0}}}{\sum }1 \\ \notag
&\geq &\frac{c_{p}\Phi ^{3/4}\left(M_{\alpha_k},M_{\alpha _{k}}\right) }{M^2_{\alpha _{k}}}M_{\alpha _{k}-1}^{2}\geq c_{p}\Phi ^{3/4}\left( M_{\alpha _{k}},M_{\alpha _{k}}\right)
\rightarrow \infty,\quad\text{as}\text{\quad }k\rightarrow \infty .
\end{eqnarray}

The proof is complete.
\end{proof}

\section{\textbf{OPEN PROBLEMS}}

In this chapter We state some open problems, which can be interesting for the researchers, who work in this area. First one reads as:
\begin{Conjecture} \label{Th2}
	Let  $f\in H_{1}^{\square}(G_m^{2}),$  $2^{-\alpha}<m/n\leq 2^{\alpha}.$ Then there exists an absolute constant $ c, $ such that
	\begin{equation*}
	\left\Vert S_{m,n}f-f\right\Vert_{{H_1^{\square}(G_m^2)}}\leq c{k^2}\omega_{H_1^{\square}(G_m^2)} \left(\frac{1}{M_k},f\right),
	\end{equation*}
	where $S_{m,n}f$ denotes $(m,n)$-th partial sum of the two-dimensional Vilenkin-Fourier series of $f.$
\end{Conjecture}
It is also interesting if we can prove that the following is true:
\begin{Conjecture}
a) Let $f\in H_{p}^{\square}(G_m^{2}),$ $2^{-\alpha }\leq m/n\leq 2^{\alpha }$ and
\begin{equation*}
\omega_{H_{1}^{\square}(G_m^{2})} \left( \frac{1}{M_{k}},f\right)=o\left( \frac{1}{k^2}\right),\text{ as\ }k\rightarrow \infty .  
\end{equation*}
Then
\begin{equation*}
\left\Vert S_{m,n}f-f\right\Vert_{H_{1}^{\square}(G_m^{2})}\rightarrow 0,\text{ as }m,n\rightarrow \infty.
\end{equation*}
	
b) (Sharpness) Let $2^{-\alpha}<m/n\leq 2^{\alpha}.$ Then there exists a martingale $f\in H_{1}^{\square}(G_m^{2}),$ such that 
\begin{equation*}
\omega_{H_{p}^{\square}(G_m^{2})} \left( \frac{1}{M_{k}},f\right)=O\left( \frac{1}{k^2}\right),\text{ as \ }k\rightarrow \infty
\end{equation*}
and
\begin{equation*}
\left\Vert S_{m,n}f-f\right\Vert _{1}\nrightarrow
0\text{ as  }m,n\rightarrow \infty,
\end{equation*}
where $S_{m,n}f$ denotes $(m,n)$-th partial sum of the two-dimensional Vilenkin-Fourier series of $f.$
\end{Conjecture}

Strong summability result for the two-dimensional Vilenkin-Fourier series in the case when $p=1$ and $2^{-\alpha }\leq k/l\leq 2^{\alpha }$ is also open problem:
\begin{Conjecture}
Let $f\in H_{1}\left(G^2_m\right).$ Then there exists an absolute constant $c,$ such that 	
\begin{equation*}
\underset{n,m\geq 2}{\sup }\frac{1}{\log n\log m}\underset{2^{-\alpha }\leq k/l\leq 2^{\alpha },\text{ }\left(k,l\right) \leq \left( n,m\right) }{\sum }\frac{\left\Vert S_{k,l}f\right\Vert_{H_1{\left( {G^2_m} \right)}}}{ kl}\leq c\left\Vert f\right\Vert _{H_1\left( {G^2_m} \right)},
\end{equation*}
where $S_{k,l}f$ denotes $(k,l)$-th partial sum of the two-dimensional Vilenkin-Fourier series of $f.$
\end{Conjecture}

We also state some interesting open problems without any conjectures
\begin{Problem} For any
$\begin{matrix}
f\in {H_p}\left( {G^2_m} \right) & (0<p\leq 1)  \\
\end{matrix},$ 
is it possible to find strong convergence theorems for partial sums $S_{m,n}$ with respect to the two-dimensional Vilenkin-Fourier series without any restriction on indexes $(m,n)?$
\end{Problem} 

\begin{Problem}
For any
$\begin{matrix}
f\in {H_p}\left( {G^2_m} \right) & (0<p\leq 1)  \\
\end{matrix},$ 
is it possible to find necessary and sufficient conditions in terms of the two-dimensional modulus of continuity of martingale 
$
f\in {H_p}\left( {G^2_m} \right)$ \ $\left( 0<p\le 1 \right),$ for which 
\[\begin{matrix}
{{\left\| {{S}_{{{k}_{j}},{{l}_{j}}}}f-f \right\|}_{{H_p}\left( {G^2_m} \right)}}\to 0, & as & j\to \infty,  \\
\end{matrix}\]
where $S_{k_j,l_j}f$ denotes $(k_j,l_j)$-th partial sum of the two-dimensional Vilenkin-Fourier series of $f?$
\end{Problem} 

\begin{Problem}
For any
$\begin{matrix}
f\in {H_p}\left( {G^2_m} \right) & (0<p\leq 1)  \\
\end{matrix},$  is it possible to find necessary and sufficient conditions for the indexes 
$\left( {k_j},{l_j} \right),$ for which 
\[\begin{matrix}
{{\left\| {{S}_{{{k}_{j}},{{l}_{j}}}}f-f \right\|}_{{H_p}\left( {G^2_m} \right)}}\to 0, & as & j\to \infty, \\
	\end{matrix}\]
where $S_{k_j,l_j}f$ denotes $(k_j,l_j)$-th partial sum of the two-dimensional Vilenkin-Fourier (Walsh-Fourier) series of $f?$
\end{Problem}

\begin{Problem} For any
$\begin{matrix}
f\in {H_p}\left( {G^2_m} \right) & (0<p\leq 1)  \\
\end{matrix},$ 
is it possible to find necessary and sufficient conditions for the indexes 
$\left( {k_j},{l_j} \right),$ for which 
\[{{\left\| \underset{j\in \mathbb{N}}{\mathop{\sup }}\,\left| {S_{k_j,l_j}}f \right| \right\|}_{{{H}_{p}}\left( {G^2_m} \right)}}\le {c_p}{{\left\| f \right\|}_{{H_p}\left( {G^2_m}\right)}},\]
where $S_{k_j,l_j}f$ denotes $(k_j,l_j)$-th partial sum of the two-dimensional Vilenkin-Fourier (Walsh-Fourier) series of $f?$
\end{Problem} 

\textbf{Acknowledgment:} The author would like to thank the referee for helpful suggestions, which absolutely improve the final version of our paper.

\end{document}